\documentclass[12pt]{article}
\usepackage{amsmath,amssymb,amsbsy,amsfonts,amsthm,latexsym,amsopn,amstext,
                                     amsopn,amstext,amsxtra,euscript,amscd}

\begin{document}

\newfont{\teneufm}{eufm10}
\newfont{\seveneufm}{eufm7}
\newfont{\fiveeufm}{eufm5}
%
%
\newfam\eufmfam
                     \textfont\eufmfam=\teneufm \scriptfont\eufmfam=\seveneufm
                     \scriptscriptfont\eufmfam=\fiveeufm

%
%
\def\frak#1{{\fam\eufmfam\relax#1}}
%


\def\bbbr{{\rm I\!R}} 
\def\bbbc{{\rm I\!C}} 
\def\bbbm{{\rm I\!M}}
\def\bbbn{{\rm I\!N}} 
\def\bbbf{{\rm I\!F}}
\def\bbbh{{\rm I\!H}}
\def\bbbk{{\rm I\!K}}
\def\bbbl{{\rm I\!L}}
\def\bbbp{{\rm I\!P}}
\newcommand{\lcm}{{\rm lcm}}
\def\bbbone{{\mathchoice {\rm 1\mskip-4mu l} {\rm 1\mskip-4mu l}
{\rm 1\mskip-4.5mu l} {\rm 1\mskip-5mu l}}}
\def\bbbc{{\mathchoice {\setbox0=\hbox{$\displaystyle\rm C$}\hbox{\hbox
to0pt{\kern0.4\wd0\vrule height0.9\ht0\hss}\box0}}
{\setbox0=\hbox{$\textstyle\rm C$}\hbox{\hbox
to0pt{\kern0.4\wd0\vrule height0.9\ht0\hss}\box0}}
{\setbox0=\hbox{$\scriptstyle\rm C$}\hbox{\hbox
to0pt{\kern0.4\wd0\vrule height0.9\ht0\hss}\box0}}
{\setbox0=\hbox{$\scriptscriptstyle\rm C$}\hbox{\hbox
to0pt{\kern0.4\wd0\vrule height0.9\ht0\hss}\box0}}}}
\def\bbbq{{\mathchoice {\setbox0=\hbox{$\displaystyle\rm
Q$}\hbox{\raise
0.15\ht0\hbox to0pt{\kern0.4\wd0\vrule height0.8\ht0\hss}\box0}}
{\setbox0=\hbox{$\textstyle\rm Q$}\hbox{\raise
0.15\ht0\hbox to0pt{\kern0.4\wd0\vrule height0.8\ht0\hss}\box0}}
{\setbox0=\hbox{$\scriptstyle\rm Q$}\hbox{\raise
0.15\ht0\hbox to0pt{\kern0.4\wd0\vrule height0.7\ht0\hss}\box0}}
{\setbox0=\hbox{$\scriptscriptstyle\rm Q$}\hbox{\raise
0.15\ht0\hbox to0pt{\kern0.4\wd0\vrule height0.7\ht0\hss}\box0}}}}
\def\bbbt{{\mathchoice {\setbox0=\hbox{$\displaystyle\rm
T$}\hbox{\hbox to0pt{\kern0.3\wd0\vrule height0.9\ht0\hss}\box0}}
{\setbox0=\hbox{$\textstyle\rm T$}\hbox{\hbox
to0pt{\kern0.3\wd0\vrule height0.9\ht0\hss}\box0}}
{\setbox0=\hbox{$\scriptstyle\rm T$}\hbox{\hbox
to0pt{\kern0.3\wd0\vrule height0.9\ht0\hss}\box0}}
{\setbox0=\hbox{$\scriptscriptstyle\rm T$}\hbox{\hbox
to0pt{\kern0.3\wd0\vrule height0.9\ht0\hss}\box0}}}}
\def\bbbs{{\mathchoice
{\setbox0=\hbox{$\displaystyle     \rm S$}\hbox{\raise0.5\ht0\hbox
to0pt{\kern0.35\wd0\vrule height0.45\ht0\hss}\hbox
to0pt{\kern0.55\wd0\vrule height0.5\ht0\hss}\box0}}
{\setbox0=\hbox{$\textstyle        \rm S$}\hbox{\raise0.5\ht0\hbox
to0pt{\kern0.35\wd0\vrule height0.45\ht0\hss}\hbox
to0pt{\kern0.55\wd0\vrule height0.5\ht0\hss}\box0}}
{\setbox0=\hbox{$\scriptstyle      \rm S$}\hbox{\raise0.5\ht0\hbox
to0pt{\kern0.35\wd0\vrule height0.45\ht0\hss}\raise0.05\ht0\hbox
to0pt{\kern0.5\wd0\vrule height0.45\ht0\hss}\box0}}
{\setbox0=\hbox{$\scriptscriptstyle\rm S$}\hbox{\raise0.5\ht0\hbox
to0pt{\kern0.4\wd0\vrule height0.45\ht0\hss}\raise0.05\ht0\hbox
to0pt{\kern0.55\wd0\vrule height0.45\ht0\hss}\box0}}}}
\def\bbbz{{\mathchoice {\hbox{$\sf\textstyle Z\kern-0.4em Z$}}
{\hbox{$\sf\textstyle Z\kern-0.4em Z$}}
{\hbox{$\sf\scriptstyle Z\kern-0.3em Z$}}
{\hbox{$\sf\scriptscriptstyle Z\kern-0.2em Z$}}}}
\def\ts{\thinspace}

\newtheorem{theorem}{Theorem}
\newtheorem{lemma}[theorem]{Lemma}
\newtheorem{claim}[theorem]{Claim}
\newtheorem{cor}[theorem]{Corollary}
\newtheorem{prop}[theorem]{Proposition}
\newtheorem{definition}{Definition}
\newtheorem{question}[theorem]{Open Question}

\def\squareforqed{\hbox{\rlap{$\sqcap$}$\sqcup$}}
\def\qed{\ifmmode\squareforqed\else{\unskip\nobreak\hfil
\penalty50\hskip1em\null\nobreak\hfil\squareforqed
\parfillskip=0pt\finalhyphendemerits=0\endgraf}\fi}

\def\cA{{\mathcal A}}
\def\cB{{\mathcal B}}
\def\cC{{\mathcal C}}
\def\cD{{\mathcal D}}
\def\cE{{\mathcal E}}
\def\cF{{\mathcal F}}
\def\cG{{\mathcal G}}
\def\cH{{\mathcal H}}
\def\cI{{\mathcal I}}
\def\cJ{{\mathcal J}}
\def\cK{{\mathcal K}}
\def\cL{{\mathcal L}}
\def\cM{{\mathcal M}}
\def\cN{{\mathcal N}}
\def\cO{{\mathcal O}}
\def\cP{{\mathcal P}}
\def\cQ{{\mathcal Q}}
\def\cR{{\mathcal R}}
\def\cS{{\mathcal S}}
\def\cT{{\mathcal T}}
\def\cU{{\mathcal U}}
\def\cV{{\mathcal V}}
\def\cW{{\mathcal W}}
\def\cX{{\mathcal X}}
\def\cY{{\mathcal Y}}
\def\cZ{{\mathcal Z}}

\newcommand{\comm}[1]{\marginpar{%
\vskip-\baselineskip 
\raggedright\footnotesize
\itshape\hrule\smallskip#1\par\smallskip\hrule}}





\hyphenation{re-pub-lished}

\def\ord{{\mathrm{ord}}}
\def\Nm{{\mathrm{Nm}}}
\renewcommand{\vec}[1]{\mathbf{#1}}

\def \F{{\bbbf}}
\def \L{{\bbbl}}
\def \K{{\bbbk}}
\def \Z{{\bbbz}}
\def \N{{\bbbn}}
\def \Q{{\bbbq}}
\def\E{{\mathbf E}}
\def\bH{{\mathbf H}}
\def\G{{\mathcal G}}
\def\O{{\mathcal O}}
\def\cS{{\mathcal S}}
\def \R{{\bbbr}}
\def\Fp{\F_p}
\def \fp{\Fp^*}
\def\\{\cr}
\def\({\left(}
\def\){\right)}
\def\fl#1{\left\lfloor#1\right\rfloor}
\def\rf#1{\left\lceil#1\right\rceil}

\def\Zm{\Z_m}
\def\Zt{\Z_t}
\def\Zp{\Z_p}
\def\Um{\cU_m}
\def\Ut{\cU_t}
\def\Up{\cU_p}

\def\ep{{\mathbf{e}}_p}

\def \Prob{{\mathrm {}}}

\def\LC{{\cL}_{C,\cF}(Q)}
\def\LCn{{\cL}_{C,\cF}(nG)}
\def\Mrs{\cM_{r,s}\(\F_p\)}

\def\Fbar{\overline{\F}_q}
\def\Fn{\F_{q^n}}
\def\En{\E(\Fn)}
\def\Sha{\mathrm{III}_p}

\def\mand{\qquad \mbox{and} \qquad}

\def\MOV{{\bf{MOV}}}


\title{{\bf Tate-Shafarevich Groups 
and Frobenius Fields of  
Reductions of Elliptic Curves}}

\author{
{\sc Igor E. Shparlinski} \\
{Dept. of Computing, Macquarie University} \\
{Sydney, NSW 2109, Australia} \\
{\tt igor@ics.mq.edu.au}}

\date{\today}
\pagenumbering{arabic}

\maketitle

\begin{abstract}
Let $\E/\Q$ be a fixed elliptic curve over $\Q$ 
which  does not have complex multiplication.
Assuming the  Generalized Riemann 
Hypothesis, 
A.~C.~Cojocaru and W.~Duke have obtained an  asymptotic formula 
for the number of primes $p\le x$ such that the reduction of 
$\E$ modulo $p$ has a trivial Tate-Shafarevich group. Recent 
results of A.~C.~Cojocaru and C.~David lead to a better
error term.  We introduce
a new argument in the scheme of the proof which gives further 
improvement. 
\end{abstract}


\section{Introduction}

Let $\E/\Q$ be a fixed elliptic curve over $\Q$ of conductor $N$, 
we refer to~\cite{Sil} for the background on elliptic curves. 
For a prime $p\nmid N$ we denote the reduction of
$\E$ modulo $p$  as $\E_p/\F_p$

As in~\cite{CojDuke}, we use
$\Sha$  to denote the {\it Tate-Shafarevich group\/} of  $\E_p/\F_p$
which is an analogue of the classical Tate-Shafarevich
group (see~\cite{Sil}) defined with respect to $\E_p$ and
the function field $\K$ of $\E_p$, that is, 
$$
\Sha= \mathrm{III}(\E_p/K),
$$
we refer to~\cite{CojDuke} for a precise definition.

Let $\pi_{TS}(x)$ be the counting function of primes $p\nmid N$ 
for which $\Sha$ is trivial. More formally,
$$\pi_{TS}(x)=  \# \{p\le x\ | \  p\nmid N, \
 \#\Sha =1\}.
$$
As usual, we also use $\pi(x)$ to denote the 
number of primes $p\le x$.

Cojocaru and Duke~\cite[Proposition~5.3]{CojDuke} 
have proved that if $\E$  does not have complex 
multiplication then under the {\it Generalized Riemann 
Hypothesis\/} (GRH) the following 
asymptotic formula
\begin{equation}
\label{eq:CD asymp}
\pi_{TS}(x) = \alpha \pi(x) + R(x) 
\end{equation}
holds for some explicitly defined constant $\alpha$
depending on $\E$, 
where 
\begin{equation}
\label{eq:CD error 1}
R(x) = O(x^{53/54+ o(1)})
\end{equation}
(hereafter implicit constants in the symbols
`$O$', `$\ll$' and `$\gg$' may depend on
$\E$). Furthermore, we have $\alpha > 0$ if and only if
$\E$ has an irrational point of order two.

The proof of~\eqref{eq:CD error 1}  is based on the square sieve of
Heath-Brown~\cite{HB1} combined with a bound of certain character
sums.  This character sum  has been estimated in a sharper way by 
Cojocaru and David~\cite[Theorem~3]{CojDav}, who also noticed that 
using their estimate in the proof of~\eqref{eq:CD error 1} 
from~\cite{CojDuke}  
reduces the error term in~\eqref{eq:CD asymp} to 
\begin{equation}
\label{eq:CD error 2}
R(x) = O(x^{41/42+ o(1)}).
\end{equation}

Here we introduce some additional element in the approach 
of~\cite{CojDuke}, which we also
combine  with  the aforementioned stronger bound of character sums
of~\cite[Theorem~3]{CojDav}, to obtain a further improvement 
of~\eqref{eq:CD error 1} and~\eqref{eq:CD error 2}. 
Namely, we obtain an extra saving
from   taking advantage of averaging over a certain
parameter $m$, which appears in the argument of Cojocaru and
Duke~\cite{CojDuke}.  To take the most out of this, we apply the bound
of double  character sums due to Heath-Brown~\cite{HB2}. 
This yields the following estimate:

\begin{theorem}
\label{thm:Sha}
Suppose $\E$  does not have complex 
multiplication and also assume that the GRH holds.
Then the asymptotic formula~\eqref{eq:CD asymp}
holds with 
$$
R(x) = O(x^{39/40+ o(1)}).
$$
\end{theorem}

%
%
%

The main goal of~\cite{CojDav} is to estimate 
$\Pi(\K,x)$ which is the number of primes $p\le x$
with $p\nmid N$ and such that a root of the 
{\it Frobenius endomorphism\/} of $\E_p/\F_p$
generates the imaginary quadratic field $\K$. 
The famous {\it Lang-Trotter conjecture\/}, which asserts
that if $\E$  does not have complex 
multiplication  then
$$
\Pi(\K,x) \sim \beta(\K) \frac{x^{1/2}}{\log x}
$$
with some constant $\beta>0$ depending on $\K$ (and on $\E$), 
remains open. However, under the GRH, the bound
\begin{equation}
\label{eq:Pi nonunif}
\Pi(\K,x) \le C(\K) \frac{x^{4/5}}{\log x} 
\end{equation}
has been given by Cojocaru and David~\cite[Theorem~2]{CojDav},
where the constant $C(\K)$ depends on $\K$ (and on $\E$).
Moreover, using the aforementioned new bound 
of  character sums, 
Cojocaru and David~\cite[Corollary~4]{CojDav}
have given a weaker, but uniform with respect to $\K$,  bound 
\begin{equation}
\label{eq:Pi unif}
\Pi(\K,x) = O\(x^{13/14}\log x\).
\end{equation} 

For real $4x \ge u > v \ge 1$, we now consider the average value 
$$
\sigma(x;u,v) = 
\sum_{\substack{u-v \le m \le u\\m~\text{squarefree}}}  \Pi(\K_m,x) 
$$
where $\K_m = \Q(\sqrt{-m})$. 
We also put
$$
\sigma(x;v,v) = \sigma(x;v).
$$ 
Clearly, the 
nonuniform bound~\eqref{eq:Pi nonunif} cannot be used 
to estimate $\sigma(x;u,v)$, while~\eqref{eq:Pi unif} 
immediately implies that uniformly over $u$, 
\begin{equation}
\label{eq:CD bound}
\sigma(x;u,v) = O\(v x^{13/14}\log x\).
\end{equation} 
Since we trivially have  $\sigma(x;u,v) \le \pi(x)$, 
the above bound is nontrivial only for $v \le x^{1/14}$.
Here we obtain a more accurate bound which remains nontrivial
for values of $v$ up to $x^{1/13 - \varepsilon}$ for 
arbitrary $\varepsilon > 0$ and sufficiently large $x$.

\begin{theorem}
\label{thm:sigma}
Suppose $\E$  does not have complex 
multiplication and also assume that the GRH holds.
Then for $4x \ge u > v \ge 1$ we have
$$
\sigma(x;u,v)  \le (v x)^{55/59+o(1)}
$$
and 
$$
\sigma(x;v)  \le v^{13/14} x^{13/14+o(1)}.
$$
\end{theorem}

It is easy to check that the first bound of 
 Theorem~\ref{thm:sigma} is nontrivial and stronger 
than~\eqref{eq:CD bound} in the range
$$
 x^{3/56+\varepsilon} \le v \le x^{4/55-\varepsilon}
$$ 
for any fixed $\varepsilon > 0$ and sufficiently large $x$.

Let 
$$
\cM(x) = \{m \in \Z \ | \ \Pi(\K_m,x) > 0\}
$$
(where as before $\K_m = \Q(\sqrt{-m})$). 
An immediate implication of~\eqref{eq:Pi unif} 
is the bound
$$
\# \cM(x) \gg \frac{x^{1/14}}{(\log x)^2}.
$$
see~\cite[Corollary~4]{CojDav}.
We now observe that the first inequality 
of Theorem~\ref{thm:sigma} implies that 
for almost all primes $p\le x$ 
the corresponding   Frobenius field 
is of discriminant at least $x^{1/13 + o(1)}$.  
In particular, we have 
$$
\max_{m \in \cM(x)} m \ge x^{1/13 + o(1)}.
$$

\section{Character Sums} 

For $p\nmid N$, we put 
$$
a_p = p + 1 -  \#\E_p(\F_p),
$$
where $\#\E_p(\F_p)$ is the number of $\F_p$-rational points of $\E_p$. When $p\mid N$,
we simply put $a_p = 1$. We recall that by the {\it Hasse bound\/}, 
$|a_p| \le 2 p^{1/2}$, see~\cite{Sil}.

We recall that the size of $\Sha$ is given by 
$$
\# \Sha = \left\{\begin{array}{ll} 
s_p^2,
 &\quad\text{if $4p - a_p^2$ is odd,}\\  
s_p^2/4,
 &\quad\text{if $4p - a_p^2$ is even,}
\end{array}\right.
$$
where 
the integer $s_p$ is uniquely defined by the relation 
$4p - a_p^2 = s_p^2 r_p$ with a squarefree integer $r_p$
(clearly  $4p - a_p^2 \equiv 0,3 \pmod 4$). 
Thus, it is natural 
to use  the square sieve~\cite{HB1} to study the 
distribution of $\#\Sha$.
This requires nontrivial bounds of sums 
with the Jacobi symbols   with $4p - a_p^2$  modulo 
products $\ell_1\ell_2$ of two distinct primes. 
Accordingly, for an odd positive integer $n$ we define
$$
U(x;n) = \sum_{p \le x} 
\(\frac{a_p^2-4p}{n}\), 
$$ 
where, as usual,  $(k/n)$ denotes the Jacobi symbol of $k$ modulo 
$n$. 

The sum has been estimated by Cojocaru, Fouvry and 
Murty~\cite{CoFoMu} and then
sharpened by Cojocaru and Duke~\cite[Proposition~4.3]{CojDuke}. 
Furthermore, when $n = \ell_1\ell_2$ is a product of two 
distinct primes, which is the only relevant case 
for this paper, Cojocaru and David~\cite[Theorem~3]{CojDav} 
give a stronger bound which we present here in the following 
form:

\begin{lemma}
\label{lem:Char}
Suppose $\E$  does not have complex 
multiplication and also assume that the GRH holds.
Then for any  real $x \ge 1$ and for any 
 distinct primes $\ell_1, \ell_2 > 3$,  we have
$$
 U(x;\ell_1 \ell_2)  = \frac{1}{(\ell_1^2-1)(\ell_2^2-1)} \pi(x)  +
O\( (\ell_1 \ell_2)^3 x^{1/2}
\log(\ell_1 \ell_2 x)\).
$$
\end{lemma}

We also need the following special case 
of the   classical {\it
Burgess\/} bound, see~\cite[Theorems~12.5]{IwKow} taken with $r=2$. 

\begin{lemma}
\label{lem:Burg}
For any real $u \ge  v \ge 1$  and an
odd square-free integer $s$, 
$$
  \sum_{ u-v \le m \le u} 
\(\frac{m}{s}\)\ll v^{1/2}s^{3/16 + o(1)} .
$$ 
\end{lemma}

As we have mentioned, a  part of 
our improvement of~\eqref{eq:CD error 1} and~\eqref{eq:CD error 2}
 comes from  bringing into the argument of~\cite{CojDuke}
the  following result of Heath-Brown~\cite{HB2}.

\begin{lemma}
\label{lem:HB}
For any real positive $X$ and $Y$ 
with $XY\to \infty$ and complex-valued function $f(m)$,
$$
   \sum_{\substack{s \le Y\\ s~\text{odd squarefree}}}
\left | \sum_{\substack{m \le X\\ m~\text{squarefree}}}
 f(m)
\(\frac{m}{s}\)\right|^2 \le (XY)^{o(1)}(X+Y)  
 \sum_{1 \le m \le X}  |f(m)|^2.
$$
\end{lemma}
 
\section{Square Multiples and Divisors of $4p-a_p^2$}

As in~\cite{CojDuke}, we define
$$
S_m(x) = \# \{p\le x\ | \   m(4p-a_p^2)  \text{ is a square}\}.
$$ 

\begin{lemma}
\label{lem:Sum Sm long}
Suppose $\E$  does not have complex 
multiplication and also assume that the GRH holds.
Then for any  real   $4x \ge u \ge v \ge 1$,  we have
$$
\sum_{\substack{u-v \le m \le u\\m~\text{squarefree}}}  S_m(x) \le 
v^{55/59}x^{55/59+o(1)}.
$$
\end{lemma}

\begin{proof} 
Fix some 
\begin{equation}
\label{eq:large z}
z \ge (\log u)^2
\end{equation}  
and assume that $x$ is  sufficiently large. 
Then by~\cite[Bound~(37)]{CojDuke} we
have
\begin{equation}
\label{eq:Sm}
S_m(x) \ll  \frac{1}{\pi(z)^2} \sum_{n \le 4 x^2}
w_m(n) \(\sum_{z \le \ell \le 2z}\(\frac{n}{\ell}\)\)^2, 
\end{equation} 
where the inner sum is taken over all primes $\ell \in [z, 2z]$ and 
$$
w_m(n) = \# \{p\le x\ | \   m(4p-a_p^2)=n\}. 
$$
We now derive
\begin{eqnarray*}  
\sum_{\substack{u-v \le m \le u\\m~\text{squarefree}}} S_m(x) &\le & 
\sum_{ u-v \le m \le u} S_m(x) \\
&\le & 
\frac{1}{\pi(z)^2}  
\sum_{ u-v \le m \le u}\sum_{n \le 4 x^2}
w_m(n) \(\sum_{z \le \ell \le 2z}\(\frac{n}{\ell}\)\)^2\\
 & =  &      \frac{1}{\pi(z)^2} \sum_{ u-v \le m \le u} \sum_{z \le
\ell_1, \ell_2\le 2z}\sum_{n \le 4 x^2}
w_m(n)\(\frac{n}{\ell_1\ell_2}\).
\end{eqnarray*}  
Separating $\pi(z)$ diagonal terms with $\ell_1  = \ell_2$, we obtain
\begin{equation}
\label{eq:Sm and Sigmas}
\sum_{ u-v \le m \le u}S_m(x) \ll    
\frac{1}{\pi(z)} \Sigma_1 +
 \frac{1}{\pi(z)^2}\Sigma_2,
\end{equation} 
where
\begin{eqnarray*}  
\Sigma_1 & = &   \sum_{ u-v \le m \le u} \sum_{n \le 4 x^2} w_m(n),\\
\Sigma_2 & = &  \sum_{ u-v \le m \le u}  \sum_{n \le 4 x^2}
\sum_{z \le \ell_1 < \ell_2\le 2z}  w_m(n)\(\frac{n}{\ell_1\ell_2}\).
\end{eqnarray*} 

We estimate the first sums trivially as
\begin{equation}
\label{eq:Sigma1}
\Sigma_1  \le     \sum_{ u-v \le m \le u} \pi(x) \le v \pi(x).
\end{equation} 

For the second sum, we note that
\begin{eqnarray*} 
 \sum_{n \le 4 x^2} w_m(n)\(\frac{n}{\ell_1\ell_2}\)
& = &\sum_{p \le x}
 \(\frac{m(4p - a_p^2)}{\ell_1\ell_2}\)\\
& = &\(\frac{-m}{\ell_1\ell_2}\) 
   \sum_{p \le x}
 \(\frac{a_p^2-4p}{\ell_1\ell_2}\)
=\(\frac{-m}{\ell_1\ell_2}\) U(x;\ell_1\ell_2) .
\end{eqnarray*}  
Thus, 
changing the order of summation, we derive
$$
\Sigma_2  = \sum_{z \le \ell_1 < \ell_2\le
2z} \sum_{ u-v \le m \le u} \(\frac{-m}{\ell_1\ell_2}\)
U(x;\ell_1\ell_2) .
$$ 
By Lemma~\ref{lem:Char}, we have 
$$ 
U(x;\ell_1\ell_2) \ll x^{1 + o(1)} (\ell_1\ell_2)^{-2} 
+ x^{1/2 + o(1)} (\ell_1\ell_2)^3 = x^{1 + o(1)} z^{-4} + x^{1/2 +
o(1)} z^6 , 
$$  
which yields the the estimate
\begin{equation}
\label{eq:Sigma2 Prelim}
\Sigma_2  \le \(x^{1 + o(1)} z^{-4} + x^{1/2 + o(1)} z^6\)  \sum_{z \le
\ell_1 < \ell_2\le 2z} \left|\sum_{ u-v \le m \le u}
\(\frac{m}{\ell_1\ell_2}\) \right| .
\end{equation}

We now apply Lemma~\ref{lem:Burg}  to derive 
from~\eqref{eq:Sigma2 Prelim} that 
\begin{equation}
\label{eq:Sigma2 Burg}
\Sigma_2  \le x^{o(1)}\(xz^{-4} + x^{1/2} z^6\) v^{1/2}z^{19/8}.
\end{equation}

Substitution of~\eqref{eq:Sigma1} and~\eqref{eq:Sigma2 Burg} 
in~\eqref{eq:Sm and Sigmas} 
leads us to the bound: 
\begin{eqnarray*}
\sum_{\substack{u-v \le m \le u\\m~\text{squarefree}}} S_m(x) &\le &
x^{o(1)}\(x v   z^{-1} +   v^{1/2}xz^{-29/8} + v^{1/2}x^{1/2}
z^{51/8}\)\\
&\le &  x^{o(1)}\( v x z^{-1} +  v^{1/2} x^{1/2} z^{51/8}\)  . 
\end{eqnarray*} 
Choosing 
$$
z =  (vx)^{4/59}
$$
(thus~\eqref{eq:large z} holds),
we conclude the proof. 
\end{proof}

For any fixed $\varepsilon> 0$, 
Lemma~\ref{lem:Sum Sm long} gives a nontrivial
estimate provided that  $ v \le x^{4/55-\varepsilon}$
uniformly over $u$. 

In the case of $u = v$, we now obtain a slightly better bound. 

\begin{lemma}
\label{lem:Sum Sm short}
Suppose $\E$  does not have complex 
multiplication and also assume that the GRH holds.
Then for any  real   $4x \ge v \ge 1$,  we have
$$
\sum_{\substack{1 \le m \le v\\m~\text{squarefree}}}  S_m(x) \le 
v^{13/14} x^{13/14+o(1)} .
$$
\end{lemma}

\begin{proof} We proceed as in the proof of 
Lemma~\ref{lem:Sum Sm long}, however, we always preserve the
condition that $m$ is square-free. Then we  can 
estimate $\Sigma_2$  by using Lemma~\ref{lem:HB}
instead of Lemma~\ref{lem:Burg}.

More precisely, 
applying the Cauchy inequality and then using 
Lemma~\ref{lem:HB} with 
$X =v$,  $Y =  4z^2 $ and $f(m)=1$, we obtain
\begin{eqnarray*}
\lefteqn{\sum_{z \le \ell_1 < \ell_2\le
2z} \left|\sum_{\substack{1 \le m \le v\\m~\text{squarefree}}}
\(\frac{m}{\ell_1\ell_2}\) \right|    \ll  \(z^2 \sum_{z \le \ell_1 <
\ell_2\le 2z} \left|\sum_{\substack{1 \le m \le
v\\m~\text{squarefree}}} \(\frac{m}{\ell_1\ell_2}\)
\right|^2\)^{1/2}  } \\
  && \qquad \qquad \qquad \qquad \qquad \le  \(x^{o(1)} vz^2\(v+z^2\)
\)^{1/2}= x^{ o(1)}\(vz + v^{1/2}z^{2}\).
\end{eqnarray*} 

We now derive from~\eqref{eq:Sigma2 Prelim} that 
\begin{equation}
\label{eq:Sigma2 HB}
\Sigma_2 \le  
 x^{o(1)}\(vxz^{-3} + v^{1/2} x z^{-2} + x^{1/2}vz^{7} +
v^{1/2}x^{1/2}z^{8}\).
\end{equation}

Substitution of~\eqref{eq:Sigma1} and~\eqref{eq:Sigma2 HB} 
in~\eqref{eq:Sm and Sigmas} 
leads us to the bound:
$$
\sum_{\substack{1 \le m \le v\\m~\text{squarefree}}} S_m(x) \le
x^{o(1)}\(vxz^{-1} +  vxz^{-5} + v^{1/2} x z^{-4} + vx^{1/2}z^{5} +
v^{1/2}x^{1/2}z^{6}\). 
$$
Clearly the second and the third terms  are both
dominated by the  first term. Hence the bound
simplifies as 
$$
\sum_{\substack{1 \le m \le v\\m~\text{squarefree}}} S_m(x) \le
x^{o(1)}\(vxz^{-1} + vx^{1/2}z^{5} + v^{1/2}x^{1/2}z^{6}\). 
$$

If we choose  
$$
z = (vx)^{1/14} 
$$
(thus~\eqref{eq:large z} holds), 
to balance the first and the third terms as $(vx)^{13/14}$,
which also gives 
$v^{19/14}x^{6/7}$ for the second term, we obtain 
$$
\sum_{\substack{1 \le m \le v\\m~\text{squarefree}}} S_m(x)\le 
x^{o(1)}\((vx)^{13/14}  +  v^{19/14}x^{6/7}\) .
$$
Clearly the bound is nontrivial only if $(vx)^{13/14}\le x$
or $v\le x^{1/13}$ in which case 
$(vx)^{13/14}  >v^{19/14}x^{6/7}$, 
thus the first term always dominates.
\end{proof}

Also as in~\cite{CojDuke}, we define
$$
D(x,y) = \sum_{y \le n \le 2x^{1/2}} \pi_n(x), 
$$
where 
$$
\pi_n(x) = \# \{p\le x\ | \ p\nmid N, \ 
n^2 \mid \#\Sha \}.
$$

This function is of independent interest. 
Our next result improves~\cite[Proposition~5.2]{CojDuke}. 

\begin{lemma}
\label{lem:Dxy}
Suppose $\E$  does not have complex 
multiplication and also assume that the GRH holds.
Then for any  real $1 \le y \le 2x^{1/2}$,  we have
$$
D(x,y)\le x^{13/7+o(1)} y^{-13/7}
$$
\end{lemma}

\begin{proof} It is easy to check that~\cite[Bound~(36)]{CojDuke}
can in fact be replaced by the following estimate
$$
D(x,y) \le x^{o(1)} 
\sum_{\substack{m \le 4x/y^2\\m~\text{squarefree}}} S_m(x)
$$
We note that this bound differs from~\cite[Bound~(36)]{CojDuke} only 
in that we still require $m$ to be squarefree. This condition
is present in all considerations which have lead
to~\cite[Bound~(36)]{CojDuke}, but is not included in that bound.
Preserving this condition does not give any advantage for the argument
of~\cite{CojDuke} but is  important for us. 
Using Lemma~\ref{lem:Sum Sm long} for $y < x^{5/12}$
and Lemma~\ref{lem:Sum Sm short} otherwise,
we obtain the result.
\end{proof}

\section{Proofs of Theorems~\ref{thm:Sha} 
and~\ref{thm:sigma}}

 As in the proof
of~\cite[Proposition~5.3]{CojDuke} we see
that for any $1 \le y\le 2x^{1/2}$ we have
$$
\pi_{TS}(x) = \alpha \pi(x) + O\(D(x,y) + x^{1/2+ o(1)} y \)
$$
where  $\alpha$ is as in~\eqref{eq:CD asymp}.  
Using the second bound of Lemma~\ref{lem:Dxy}, we derive
$$
\pi_{TS}(x) = \alpha \pi(x) + 
O\( x^{13/7+o(1)} y^{-13/7} + x^{1/2+ o(1)} y \),
$$
and then selecting $y = x^{19/40}$, 
we conclude the proof of Theorem~\ref{thm:Sha}.

To proof Theorem~\ref{thm:sigma}, as in~\cite{CojDav}, 
we note that 
$$
\sigma(x;u,v) \le \sum_{\substack{u-v \le m \le
u\\m~\text{squarefree}}} S_m(x).
$$
Now Lemmas~\ref{lem:Sum Sm long} and~\ref{lem:Sum Sm short}
imply the result.

\section{Remarks}

Under some additional assumptions, 
Cojocaru and David~\cite[Theorem~3]{CojDav}
give sharper bounds on the 
error term in the asymptotic formula of
Lemma~\ref{lem:Char}. In turn, this leads to further sharpening the
bound of  Theorem~\ref{thm:Sha} (under the same
additional assumptions).

We also note that, Lemma~\ref{lem:Dxy} shows
that under the GRH the bound $\# \Sha \le x^{12/13 + o(1)}$
holds for all but $o(\pi(x))$ primes $p\le x$. 

It would be very interesting to obtain an unconditional 
proof of the asymptotic formula~\eqref{eq:CD asymp}
with $R(x) = o(\pi(x))$. 

In fact, it is  possible to obtain an unconditional
version of Lemma~\ref{lem:Char}. However, it seems to be 
too weak to leads to an asymptotic formula for 
$\pi_{TS}(x)$. Indeed, to use this unconditional
version, one  needs a
nontrivial estimate on $D(x,y)$ 
for rather small values of $y$. 
Although the approach of Lemma~\ref{lem:Dxy} 
admits an unconditional version, it seems
highly unlikely that
without some principally new ideas one can 
obtain an unconditional asymptotic formula for 
$\pi_{TS}(x)$.


\section*{Acknowledgement}

The author is grateful to  Alina Cojocaru, 
Bjorn Poonen and Joseph  Silverman for useful 
comments.

This work was supported in part by ARC grant
DP0556431.

\end{document}